\documentclass[12pt, reqno]{amsart}
\usepackage{amsmath, amsthm, amscd, amsfonts, amssymb, graphicx, color, mathrsfs}
\usepackage[bookmarksnumbered, colorlinks, plainpages]{hyperref}
\usepackage[all]{xy} 
\usepackage{slashed}

\newtheorem{theorem}{Theorem}[section]
\newtheorem{lemma}[theorem]{Lemma}

\theoremstyle{definition}

\theoremstyle{remark}
\newtheorem{remark}[theorem]{Remark}
\numberwithin{equation}{section}

\begin{document}
\setcounter{page}{1}

\title[ Oscillating  singular integral  on compact Lie groups revisited]{Oscillating  singular integral operators on compact Lie groups revisited}

\author[D. Cardona]{Duv\'an Cardona}
\address{
  Duv\'an Cardona:
  \endgraf
  Department of Mathematics: Analysis, Logic and Discrete Mathematics
  \endgraf
  Ghent University, Belgium
  \endgraf
  {\it E-mail address} {\rm duvanc306@gmail.com, duvan.cardonasanchez@ugent.be}
  }
  
\author[M. Ruzhansky]{Michael Ruzhansky}
\address{
  Michael Ruzhansky:
  \endgraf
  Department of Mathematics: Analysis, Logic and Discrete Mathematics
  \endgraf
  Ghent University, Belgium
  \endgraf
 and
  \endgraf
  School of Mathematical Sciences
  \endgraf
  Queen Mary University of London
  \endgraf
  United Kingdom
  \endgraf
  {\it E-mail address} {\rm michael.ruzhansky@ugent.be, m.ruzhansky@qmul.ac.uk}
  }

\thanks{The authors are supported  by the FWO  Odysseus  1  grant  G.0H94.18N:  Analysis  and  Partial Differential Equations and by the Methusalem programme of the Ghent University Special Research Fund (BOF)
(Grant number 01M01021). Michael Ruzhansky is also supported  by EPSRC grant 
EP/R003025/2.
}


\begin{abstract} In \cite[Theorem 2']{Fefferman1970} Charles Fefferman has proved the weak (1,1) boundedness for a class of oscillating singular integrals that includes the oscillating spectral multipliers of the Euclidean Laplacian $\Delta,$ namely, operators of the form 
\begin{equation}
  T_{\theta}(-\Delta):=  (1-\Delta)^{-\frac{n\theta}{4}}e^{i (1-\Delta)^{\frac{\theta}{2}}},\,0\leq \theta <1.
\end{equation}
The aim of this work is to extend Fefferman's result  to oscillating singular integrals  on any arbitrary compact Lie group. We also consider applications to oscillating spectral multipliers of the Laplace-Beltrami operator.  The proof of our main theorem illustrates the delicate relationship between the condition on the kernel of the operator, its Fourier transform (defined in terms of the representation theory of the group) and the microlocal/geometric properties of the group.
\end{abstract} 

\maketitle

\tableofcontents
\allowdisplaybreaks

\section{Introduction}

\subsection{Outline}
This work deals with the weak (1,1) boundedness of oscillating singular integrals on compact Lie groups by extending to the non-commutative setting the paradigm of the oscillating singular integrals introduced by Fefferman and Stein in \cite{Fefferman1970,FeffermanStein1972}, which among other things, generalises some classical conditions introduced by Calder\'on and Zygmund \cite{CalderonZygmund1952} and  H\"ormander \cite{Hormander1960}.
One of the novelties of this work is that we make use of the geometric properties of the group, of  its microlocal analysis, and of its representation theory.

On compact Lie groups, oscillating singular integrals arose as generalisations of the  oscillating spectral multipliers, which are linear operators defined by the spectral calculus via
\begin{equation}\label{Spectral:LaplaceBO}
  T_{\theta}(\mathcal{L}_G):=  (1+\mathcal{L}_G)^{-\frac{n\theta}{4}}e^{i (1+\mathcal{L}_G)^{\frac{\theta}{2}}},\,0\leq \theta <1,
\end{equation}where $\mathcal{L}_G:=-(X_1^2+\cdots X_n^2)$ is the positive Laplace-Beltrami operator on a compact Lie group $G.$

The problem of finding boundedness criteria for Fourier multipliers (left-convolution operators) has been widely  investigated for a long time. On a compact Lie group $G,$ the problem has been considered for central operators by N. Weiss \cite{NWeiss}, Coifman and Weiss \cite{CoifmanWeiss2},   Cowling and Sikora \cite{CowlingSikora} in the case of $G=\textnormal{SU}(2),$ and Chen and Fan \cite{Chen}. However, only  symbol criteria for general Fourier multipliers (and for pseudo-differential operators) on compact Lie groups, making use of the difference structure of the unitary dual $\widehat{G}$ of $G,$ were firstly proved by the second author and Wirth in \cite{Ruz}, and further generalisations were established in  \cite{Cardona2,CR20,CR211,RuzhanskyDelgado2017} and in the works \cite{CardonaDelgadoRuzhansky,CRgraded,FR,Guorong} for the setting of graded Lie groups. We refer the reader to \cite{alexo,CR20,Sikora} and \cite{MartiniMullerNicolussi} for an extensive list of references on the subject as well as for a historical perspective, in particular in the setting of central operators. In particular, the work \cite{RobertoBramati} introduces an interesting family of multipliers covering e.g. the H\"ormander-Mihlin condition.

Before presenting the contributions of this paper, let us review the classical result due to Charles Fefferman which is of central interest for this work,  see \cite[Page 23]{Fefferman1970}. In the Euclidean setting,  a convolution operator $T:f\mapsto f\ast K,$ $f\in C^{\infty}_0(\mathbb{R}^n),$ with $K$ being a distribution of compact support and locally integrable outside the origin, 
and satisfying the conditions
\begin{equation}\label{decay}
    |\widehat{K}(\xi)|=O((1+|\xi|)^{-\frac{n\theta}{2}}),\quad 0\leq \theta<1,
\end{equation}and 
\begin{equation}\label{FeffCond}
   [K]_{H_{\infty,\theta}}:=  \sup_{0<R\leq 1}\Vert\, \smallint\limits_{|x|\geq 2R^{1-\theta}}|K(x-y)-K(x)|dx \Vert_{L^\infty(B(0,R),\,dy)}  <\infty,
\end{equation}is called of oscillating type. Here, $\widehat{K}:=\mathscr{F}_{\mathbb{R}^n}K$ denotes the Fourier transform of the distribution $K.$ The $L^p$-properties for convolution operators of this form were firstly studied by Hardy \cite{Hardy1913}, 
Hirschman \cite{Hirschman1956} and Wainger  \cite{Wainger1965}, in the particular scenario of Fourier multipliers with symbols  of the form
\begin{equation*}
    \widehat{K}(\xi)={(1+|\xi|^2)^{-\frac{n\theta}{4}}}{e^{i(1+|\xi|^2)^{\frac{\theta}{2}}}},\,\,\quad 0<\theta<1.
\end{equation*}The corresponding operators  are oscillating multipliers of the positive Euclidean Laplacian $-\Delta=\mathcal{L}_{\mathbb{R}^n},$ namely, operators of the form 
\begin{equation}\label{Spect:multu}
  T_{\theta}(-\Delta):=  (1-\Delta)^{-\frac{n\theta}{4}}e^{i (1-\Delta)^{\frac{\theta}{2}}},\,0\leq \theta <1.
\end{equation}
One of the fundamental contributions to the subject, came with the work   \cite{Fefferman1970} by Charles Fefferman where he established the   weak (1,1)  boundedness of a class of oscillating singular integrals that includes the oscillating spectral multipliers of the Laplacian \eqref{Spect:multu}. It is also important to mention that Fefferman and Stein, when extending the theory of Hardy spaces to Euclidean spaces of several variables, observed that the theory of oscillating singular integrals in \cite{Fefferman1970} also begets bounded operators from the Hardy space $H^1$ into $L^1,$ see \cite{FeffermanStein1972} for details. 

We also observe that although it does  not appear in the hypotheses of Theorem 2' of \cite{Fefferman1970}, Fefferman has assumed  in the proof of such a statement (see \cite[Page 23]{Fefferman1970}) that the support of $K$ is small. For instance, his assumption says that $ \textnormal{diam}(\textnormal{supp}(K))<1.$ This assumption in the Euclidean setting is not restrictive because one can use the natural dilation structure of $\mathbb{R}^n,$ and the argument in \cite[Page 23]{Fefferman1970} to reduce the analysis of oscillating convolution kernels with compact support of arbitrary size to distributions with small support.

\subsection{Fefferman's approach}
Next, we review the fascinating technique developed by Fefferman in his {\it Acta}'s paper  \cite{Fefferman1970} for the proof of the aforementioned weak (1,1) estimate. As it was pointed out by  Stein in \cite[Page 1257]{AdHonorem}, such a technique became a subject of much wider interest, as
it was adapted to various other problems e.g. by S. Chanillo,
M. Christ, Rubio de Francia, Seeger,  Sogge and Stein \cite{SSS}, etc.

For this, let us fix a  convolution operator $T$ with kernel $K$ satisfying \eqref{decay} and \eqref{FeffCond} and assume that its support is small, for instance, assume that 
\begin{equation}\label{diam:Fefferman}
    \textnormal{diam}(\textnormal{supp}(K))<1.
\end{equation}
To begin with he decomposed an arbitrary function $f\in L^{1}(\mathbb{R}^n),$ for a fixed $\alpha>0,$ according to the Calder\'on-Zygmund decomposition theorem, in the standard way, that is $f=g+b,$ $b:=\sum_{j}b_j,$ where the $L^2$-norm of $g$ is bounded from above by $\alpha,$ and the $b_j$'s are supported on disjoint dyadic cubes $I_j$'s in such a way that the following two properties are  satisfied
\begin{equation}
   \frac{1}{|I_j|}\smallint\limits_{I_j}|b_j(x)|dx\sim \alpha, \quad \smallint\limits_{I_j}b_j(x)dx=0. 
\end{equation}
When estimating $|\{x\in \mathbb{R}^n:|Tf(x)|>\alpha\}|,$ he started with the standard inequality
$$|\{x\in \mathbb{R}^n:|Tf(x)|>\alpha\}|\leq |\{x\in \mathbb{R}^n:|Tb(x)|>\alpha/2\}|+|\{x\in \mathbb{R}^n:|Tg(x)|>\alpha/2\}|. $$ Then, he observed that
$T(g)$ can be estimated from above with the $L^1$-norm of $f$ using the $L^2$-theory. As for the estimate of $T(b_j),$ the main idea  in Fefferman's argument was the consideration of a family of measurable sets $I_{j}^*$ such that,
\begin{equation}
    \textnormal{center}[I_j]=\textnormal{center}[I_j^*],\quad \textnormal{diameter}[I_j^*]\sim\textnormal{diameter}[I_j]^{1-\theta}.
\end{equation}By making a suitable geometric reduction of the support of $K,$ he proved that only cubes $I_j$ with $\textnormal{diam}[I_j]<1$ were significant when estimating $T(b_j).$ Moreover, the contribution of $T(b_j)$ outside the set $I_j^{*}$ can be handled in view of the kernel condition \eqref{FeffCond}.

About the critical contribution of $T(b_j)$ in the set $I_{j}^*\setminus I_j,$ he introduced a nice replacement $\tilde{b}_{j}$ of $b_j,$ and another one $\tilde{b}=\sum_j \tilde{b}_{j}$ of $b,$ and he was able to prove that
\begin{equation}
    \Vert T(\tilde{b})\Vert_{L^2}\lesssim\Vert(-\Delta)^{-\frac{n\theta}{4}}\tilde{b}\Vert_{L^2}\lesssim \alpha\Vert f\Vert_{L^1},
\end{equation} by
combining the $L^2$-theory, a suitable decomposition of the function $(-\Delta)^{-\frac{n\theta}{4}}\tilde{b}=F_1+F_2,$ and the Sobolev inequality. Putting all these estimates together he proved the weak (1,1) type of $T.$

\subsection{Compact Lie groups setting}
It is natural to investigate whether these results for oscillating singular integrals on $\mathbb{R}^n$ can be extended to more general manifolds. In particular, in view of the behaviour of the Fourier transform of the kernel \eqref{decay} one can analyse the problem in the setting of compact Lie groups where the group Fourier transform is defined in terms of the representation theory of the group.  

On a compact Lie group $G,$ an oscillating singular integral is a left-invariant operator $T:C^\infty(G)\rightarrow \mathscr{D}'(G)$ with right convolution kernel $K$ satisfying the  estimate
\begin{equation}\label{CZ:groups}
        [K]_{H_{\infty,\theta}(G)}:=\sup_{0<R\leq 1}\Vert \smallint\limits_{|x|\geq 2R^{1-\theta}}|K(y^{-1}x)-K(x)|dx\Vert_{L^\infty(B(0,R),\,dy)} <\infty,
    \end{equation}for some $0\leq \theta<1.$ We have denoted by $|x|=d(x,e)$ the geodesic distance from $x\in G$ to the neutral element $e\in G.$

In order to study the boundedness of oscillating singular integrals in the non-commutative setting, one could adopt the argument of  Fefferman and to study its extension to compact Lie groups. Nevertheless, when extending such a  construction we found some  obstructions related with the structure of the group $G.$ More precisely, in Fefferman's argument, the $\tilde{b}_j's$ are defined by the convolution $b_j\ast \phi_{j},$ where roughly speaking, any $\phi_j$ is given by 
$$ \phi_j(y):=\textnormal{diam}[I_j]^{-n/(1-\theta)}\phi(\textnormal{diam}[I_j]^{-1/(1-\theta)}\cdot y),\quad \smallint \phi=1,\quad \textnormal{supp}[\phi]\subset\{x:|x|\leq 1\}.  $$ By taking into account the geometry of the group, we observe that there is not a global action $\mathbb{R}^{+}\times G\rightarrow G$ of $\mathbb{R}^{+}$ into $G,$ that allows us to define the elements $r\cdot y\in G,$ for $r>0$ and $y\in G,$ in a compatible way with Fefferman's argument.  Of particular interest are the dilations factors $$r=\textnormal{diam}[I_j]^{-1/(1-\theta)}>1.$$ Indeed, only on small neighborhoods $\nu$ of $0\in \mathfrak{g},$ and $\nu'$ of the  neutral element $e\in G,$ the exponential mapping $\textnormal{exp}:\nu\rightarrow \nu'$ is a diffeomorphism, and a local family of dilations $D_{r}:\nu\rightarrow \nu,$ with $0<r\leq 1,$ denoted by  $y\in \nu\mapsto D_{r}(y)=r\cdot y ,$ can be defined via $r\cdot y=\exp(r\exp^{-1}(y)).$ One can translate these dilations to small neighborhoods of any point of the group by using the multiplication operation on $G$, but, still, the case of dilations by factors $r>1$ cannot be covered with this local construction.

However, inspired by the approach developed by Coifman and Weiss \cite{CoifmanWeiss} and by Coifman and De Guzm\'an \cite{CoifmandeGuzman}, we observe that for the scenario of compact Lie groups  by taking the family of functions
\begin{equation}\label{jump:function}
    \phi_{j}(y)=\frac{1}{|B(e,R_j)|}1_{B(e,R_j)},\quad R_j\sim \textnormal{diam}[I_j]^{1/(1-\theta)},
\end{equation}will provide the necessary properties for extending  Fefferman's argument.

Since our analysis is local we will assume the following hypothesis:
\begin{itemize}
    \item[(H):] Assume that $K$ is a distribution of compact support and that its diameter is small enough, for instance, we assume that $\textnormal{supp}(K)$ is  contained in a neighborhood $\nu$ of the identity $e\in G,$ in  such a way  that  $\exp:\nu'=\exp^{-1}(\nu)\rightarrow \nu$  is a diffeomorphism, and that $\textnormal{diam}(\textnormal{supp}(K))<c,$ where $0<c\leq 1.$ 
\end{itemize}

The main theorem of this work is the following. Here $\widehat{G}$ denotes the unitary dual of $G$ and the elliptic weight $\langle\xi\rangle,$ $[\xi]\in \widehat{G},$ is defined in terms of the spectrum of the Laplacian on $G,$ see Section \ref{preliminaries} for details. 

\begin{theorem}\label{main:th}Let $G$ be a compact Lie group of dimension $n,$ and let  $T:C^\infty(G)\rightarrow\mathscr{D}'(G)$  be a left-invariant operator with right-convolution kernel $K\in L^1_{\textnormal{loc}}(G\setminus\{ e\})$ satisfying the small support condition $\textnormal{(H)}.$ Let us consider that for  $0\leq \theta<1,$ $K$   satisfies the group Fourier transform condition
\begin{equation}\label{LxiG}
  \exists C>0,\quad \forall [\xi]\in \widehat{G},\quad   \Vert \widehat{K}(\xi)\Vert_{\textnormal{op}}\leq C\langle \xi\rangle^{-\frac{n\theta}{2}},
\end{equation}
and the oscillating  H\"ormander condition
    \begin{equation}\label{GS:CZ:cond}
        [K]_{H_{\infty,\theta}(G)}:=\sup_{0<R\leq 1}\sup_{|y|\leq R} \smallint\limits_{|x|\geq 2R^{1-\theta}}|K(y^{-1}x)-K(x)|dx <\infty.
    \end{equation}   Then  $T$ admits an extension  of weak $(1,1)$ type. 
\end{theorem}
 
\begin{remark}
 It was proved in \cite{CR20221} that under the hypothesis in Theorem \ref{main:th},  $T$ is bounded from the Hardy space $H^1(G)$ into $L^1(G).$ So, in the case of compact Lie groups our main Theorem \ref{main:th} together with the results in \cite{CR20221} provide a complete perspective on the subject, related with the H\"ormander condition in \eqref{GS:CZ:cond} for all $0\leq \theta <1$.  We observe that the case $\theta=1$ is outside of the analysis in our approach. It is related to the wave operator for the Laplace-Beltrami operator and in view of the weak (1,1) estimate for Fourier integral operators by Tao \cite{Tao}, it is expected that the  right decay condition in \eqref{LxiG} for $\theta=1$ could be the one with the order $-(n-1)/2$ (instead of $-n/2$).
\end{remark} 
\begin{remark}
 The condition that the support of the distribution $K$ is contained in a neighborhood $\nu$ of the identity $e\in G,$ in  such a way  that  that  $\exp:\nu'=\exp^{-1}(\nu)\rightarrow \nu=\exp(\nu)$  is a diffeomorphism guarantees a local analysis. 
\end{remark}
\begin{remark}
We also observe  that the relevant contribution of Theorem \ref{main:th} is the case where $0<\theta <1.$ Indeed, for $\theta=0,$ the Fourier transform condition in \eqref{LxiG} can be replaced for the $L^2$-boundedness of the operator $T,$ and the condition in \eqref{GS:CZ:cond} is reduced to the {\it{ smoothness H\"ormander condition}} which can be analysed in the setting of the Calder\'on-Zygmund theory on spaces of homogeneous type in the sense of Coifman and Weiss \cite{CoifmanWeiss}. 
\end{remark}
\begin{remark}One of the differences between the approach due to Fefferman  in \cite{Fefferman1970} and our proof of Theorem \ref{main:th} lies in our application of the
 $L^2$-boundedness of the operator $(-\mathcal{L}_G)^{z}$ when $\textnormal{Re}(z)<0,$ in view of the Stein identity (see \cite[Page 58]{SteinIdentity})    $$(-\mathcal{L}_G)^{z}=\left(-\mathcal{L}_G+\textnormal{Proj}_{\textnormal{Ker}(-\mathcal{L}_G)}\right)^{z}-\textnormal{Proj}_{\textnormal{Ker}(-\mathcal{L}_G)}.$$ See e.g.  Remark 2.4 of \cite{RuzhanskyWirth2015}. Indeed, we have that $(-\mathcal{L}_G)^{z}$ is a pseudo-differential operator on $G$ of order $\textnormal{Re}(z)<0.$ Note that, this argument does not apply in the case of $\mathbb{R}^n$ because the spectrum of the Euclidean Laplacian is continuous and the negative powers of $-\Delta_{\mathbb{R}^n}$ are not pseudo-differential operators, while the powers of  $1-\Delta_{\mathbb{R}^n}$ are pseudo-differential operators with symbols in the Kohn-Nirenberg classes, see e.g. Taylor \cite{Taylorbook1981}. The boundedness of the operator $(-\mathcal{L}_G)^{z}$  on $L^2$ allows us to apply the microlocal analysis of $G.$
\end{remark}
\begin{remark}
In view of the Hopf-Rinow theorem the exponential mapping on $G$ viewed as a Riemannian manifold agrees with the exponential mapping on $G,$ considered as a closed sub-group of (the linear group of unitary matrices) $\textnormal{U}(N)$ for $N$ large enough. In particular  on small neighborhoods $\nu$ of $0\in \mathfrak{g},$ and $\nu'$ of the  neutral element $e\in G,$ the exponential mapping $\textnormal{exp}:\nu\rightarrow \nu'$ is a diffeomorphism. This geometric property of the group will be involved in the proof of our main Theorem \ref{main:th}. 
\end{remark}
\begin{remark}
Clearly $G$ being a compact topological space  is a homogeneous space in the sense of Coifman and Weiss \cite{CoifmanWeiss}. The geometric measure theory of the group will be used due to the existence of Calder\'on-Zygmund type decompositions for any $f\in L^1(G),$ see \cite[Pages 73-74]{CoifmanWeiss}.
\end{remark}
\begin{remark}
We observe that using the theory of Coifman and De Guzm\'an \cite{CoifmandeGuzman},  the second author and Wirth in \cite{Ruz} have proved H\"ormander-Mihlin criteria for Fourier multipliers on compact Lie groups. In that setting, as expected, it was a non-trivial fact that the operators satisfy Calder\'on-Zygmund type conditions  (corresponding to the case $\theta=0$  in Theorem \ref{main:th}). 
\end{remark}
\begin{remark}
For limited-range versions of the Calder\'on-Zygmund theorem we refer the reader to  Baernstein and Sawyer \cite{BS}, Carbery \cite{Carbery}, Seeger \cite{Seeger}, and  Grafakos, Honz\'ik, 
Ryabogin \cite{Grafakosetal}. We observe that  the case in the Euclidean setting, the case $\theta=0$ includes Fourier multipliers satisfying H\"ormander-Mihlin conditions. Improvements for the H\"ormander condition and conditions of H\"ormander-Mihlin type in the Euclidean framework  can be found in Grafakos \cite{Grafakos}. We refer the reader to \cite{CR20222} for the boundedness properties of oscillating singular integrals on $\mathbb{R}^n.$
\end{remark}
This paper is organised as follows. In the short Section \ref{preliminaries} we present the preliminaries on the Fourier analysis on compact Lie groups. Our main theorem is proved in Section \ref{proof:section}. Finally, some examples are given in the case of the torus $\mathbb{T}^n,$ $\textnormal{SU}(2)\cong \mathbb{S}^3,$ and for oscillating Fourier multipliers in Section \ref{Examples:t:su2:OFM}.

\section{Group Fourier transform on compact Lie groups}\label{preliminaries}
In this section we present some preliminaries about the Fourier analysis on compact Lie groups, for this we follow the book \cite[Part III]{Ruz}.

Let $G$ be a compact Lie group and let $dx$  be its normalised left-invariant Haar measure. To define the group Fourier transform let us define the unitary dual $\widehat{G}$ of  $G.$ One reason for this is that the Fourier inversion formula becomes a series over $\widehat{G}.$ Let us denote by  $\xi$  a strongly continuous, unitary and irreducible  representation of $G,$ this means that,
\begin{itemize}
    \item $\xi\in \textnormal{Hom}(G, \textnormal{U}(H_{\xi})),$ for some finite-dimensional vector space $H_\xi\cong \mathbb{C}^{d_\xi},$ i.e. $\xi(xy)=\xi(x)\xi(y)$ and for the  adjoint of $\xi(x),$ $\xi(x)^*=\xi(x^{-1}),$ for every $x,y\in G.$
    \item The map $(x,v)\mapsto \xi(x)v, $ from $G\times H_\xi$ into $H_\xi$ is continuous.
    \item For every $x\in G,$ and $W_\xi\subset H_\xi,$ if $\xi(x)W_{\xi}\subset W_{\xi},$ then $W_\xi=H_\xi$ or $W_\xi=\emptyset.$
\end{itemize} 
In view of the compactness of the group, any strongly continuous unitary representation on $G$ is continuous. 
Let $\textnormal{Rep}(G)$ be the set of unitary, continuous and irreducible representations of $G.$

One can define an equivalence relation on $\textnormal{Rep}(G)$ as follows: {{
\begin{equation*}
    \xi_1\sim \xi_2,\,\xi_1,\xi_2\in \textnormal{Rep}(G)\,\,\Longleftrightarrow\, \exists A\in \textnormal{End}(H_{\xi_1},H_{\xi_2}):\,\,\forall x\in G,\, A\xi_{1}(x)A^{-1}=\xi_2(x). 
\end{equation*}}} With respect to this equivalence relation the quotient
$$
    \widehat{G}:={\textnormal{Rep}(G)}/{\sim},
$$is the unitary dual of $G.$

The Fourier transform $\widehat{f}$ of a distribution $f\in \mathscr{D}'(G)$ is defined via,
\begin{equation*}
    \widehat{f}(\xi)\equiv (\mathscr{F}f)(\xi):=\smallint\limits_{G}f(x)\xi(x)^*dx,\,[\xi]\in \widehat{G}.
\end{equation*}  The Fourier inversion formula
\begin{equation*}
    f(x)=\sum_{[\xi]\in \widehat{G}}d_\xi\textnormal{\textbf{Tr}}[\xi(x)\widehat{f}(\xi)]
\end{equation*} holds if for example $f\in L^1(G).$ In this setting, the Plancherel theorem states that  $\mathscr{F}:L^2(G)\rightarrow L^2(\widehat{G})$ is an isometry, where the inner product on $L^2(\widehat{G})$ is defined via
\begin{equation}
    (f,g)_{L^2(\widehat{G})}:=\sum_{[\xi]\in \widehat{G}}d_\xi\textnormal{\textbf{Tr}}[\widehat{f}(\xi)\widehat{g}(\xi)^*].
\end{equation}So, the Plancherel theorem takes the form
\begin{equation}
    \Vert f\Vert_{L^2(G)}^2=\sum_{[\xi]\in \widehat{G}}d_\xi\|\widehat{f}(\xi)\|_{\textnormal{HS}}^2.
\end{equation}We have denoted by $\|\cdot \|_{\textnormal{HS}}$ the usual Hilbert-Schmidt norm on every representation space $H_\xi.$
\begin{remark}\label{reamrk:eigenvalues}
In the symbol condition \eqref{LxiG}, $$\textnormal{Spect}((1+\mathcal{L}_G)^{\frac{1}{2}}):=\{\langle \xi\rangle:[\xi]\in \widehat{G}\},$$ is the system of eigenvalues of the Bessel potential operator $(1+\mathcal{L}_G)^{\frac{1}{2}}$ associated to the positive  Laplacian on $G,$ which can be defined as follows.      Let $X=\{X_1,\cdots, X_n\}$ be an orthonormal  basis of its Lie algebra $\mathfrak{g}$ with respect to the unique (up to by a constant factor) bi-invariant Riemannian metric $g=(g_x(\cdot,\cdot))_{x\in G}$ on $G.$ The Laplace-Beltrami operator is defined by (minus) the sum of squares
\begin{equation*}
    \mathcal{L}_G=-\sum_{j=1}^{n}X_j^2,
\end{equation*} and it is independent of the choice of the vector space basis $X$ of $\mathfrak{g}.$ 
Note that $\mathcal{L}_G$ is an elliptic operator on $G,$ which admits a self-adjoint extension on $L^2(G,dx).$ Such an extension is an elliptic operator and its spectrum is a discrete set $\{\lambda_{[\xi]}:[\xi]\in \widehat{G}\}$ that can be enumerated by the unitary dual $\widehat{G}$ of the group. By the spectral mapping theorem one has that $\langle \xi\rangle=(1+\lambda_{[\xi]})^{\frac{1}{2}},$ $[\xi]\in \widehat{G}.$ 
\end{remark}
\begin{remark}\label{Rem:Four:Conf}Observe that in view of the Plancherel theorem, the hypothesis \eqref{LxiG}  in Theorem \ref{main:th} implies that the operator $(1+\mathcal{L}_G)^{\frac{n\theta}{4}}T$ admits a bounded extension on $L^2(G
)$. Indeed, the Plancherel theorems implies that
\begin{equation}
     \Vert (1+\mathcal{L}_G)^{\frac{n\theta}{4}}Tf\Vert_{L^2(G)}^2=\sum_{[\xi]\in \widehat{G}}d_\xi\|\langle\xi\rangle^{\frac{n\theta}{2}}\widehat{K}(\xi)\widehat{f}(\xi)\|_{\textnormal{HS}}^2\leq\sum_{[\xi]\in \widehat{G}}d_\xi\|\langle\xi\rangle^{\frac{n\theta}{2}}\widehat{K}(\xi)\|^2_{\textnormal{op}}\|\widehat{f}(\xi)\|_{\textnormal{HS}}^2
\end{equation}and in view of  \eqref{LxiG} we have that
$$ \Vert (1+\mathcal{L}_G)^{\frac{n\theta}{4}}Tf\Vert_{L^2(G)}^2\leq C^2\sum_{[\xi]\in \widehat{G}}d_\xi\|\widehat{f}(\xi)\|_{\textnormal{HS}}^2\lesssim \Vert f\Vert_{L^2(G)}^2 , $$ proving the boundedness of $(1+\mathcal{L}_G)^{\frac{n\theta}{4}}T$ on $L^2(G).$

\end{remark}

\section{Proof of the main theorem}\label{proof:section}

In this section we are going to prove that for a singular integral operator $T$ satisfying \eqref{decay} and \eqref{GS:CZ:cond}  there is a constant $C>0,$ such that \begin{equation}\label{To:proof}
 \Vert T f\Vert_{L^{1,\infty}(G)}:= \sup_{\alpha>0}  \alpha|\{x\in G:|Tf(x)|>\alpha\}|\leq C\Vert f\Vert_{L^1(G)},
\end{equation}with $C$ independent of $f.$ The positive Laplacian on $G$ or on $\mathbb{R}^n$ will be denoted by $\mathcal{L}_G$ and by $\mathcal{L}_{\mathbb{R}^n},$ respectively. The Euclidean Fourier transform on $\mathbb{R}^n$ of a function $V\in L^{1}(\mathbb{R}^n),$ will be denoted by $\mathscr{F}_{\mathbb{R}^n}V(\Xi)=\smallint_{\mathbb{R}^n}e^{-i2\pi X\cdot\Xi }V(X)dX,$ and it will be used at the end of this section.

For this fix $f\in L^1(G),$ and let us consider its Calder\'on-Zygmund decomposition, see Coifman and Weiss \cite[Pages 73-74]{CoifmanWeiss}. So, for any $\gamma,\alpha>0, $ such that

\begin{equation}\label{alphagamma}
\alpha\gamma>    \frac{1}{|G|}\smallint\limits_{G}|f(x)|dx,
\end{equation}
we can have the decomposition  $$f=g+b=g+\sum_{j}b_j$$ where the following properties are satisfied: 
\begin{itemize}
    \item[(1)] $\Vert g \Vert_{L^\infty}\lesssim_{G} \gamma \alpha$ and $\Vert g\Vert_{L^1}\lesssim_{G} \Vert f\Vert_{L^1}.$
    \item[(2)] The $b_j$'s are supported in  open balls $I_j=B(x_j,r_j)$ where they satisfy the cancellation property
    \begin{equation}
        \smallint\limits_{I_j}b_j(x)dx=0.
    \end{equation}
     \item[(3)] Any component $b_j$ satisfies the $L^1$-estimate 
     \begin{equation}
        \Vert b_j\Vert_{L^1}\lesssim_{G} (\gamma \alpha)|I_j|.
\end{equation}
    \item[(4)] The sequence $\{|I_{j}|\}_{j}\in \ell^1$ and
    \begin{equation}\label{I:j}
    \sum_{j} |I_j|\lesssim_{G} (\gamma\alpha)^{-1}    \Vert f\Vert_{L^1}.   
    \end{equation}
   \item[(5)] 
    $$\Vert b\Vert_{L^1}\leq\sum_j  \Vert b_j\Vert_{L^1}\lesssim_{G} \Vert f\Vert_{L^1}.$$
    \item[(6)] There exists $M_0\in \mathbb{N},$ such that any point $x\in G$ belongs at most to $M_0$ balls of the collection $I_j.$
\end{itemize}
 First of all note that given an altitude $\alpha\gamma>0,$ the  Calder\'on-Zygmund decomposition above works in the case where $\alpha\gamma>\frac{1}{|G|}\smallint|f(x)|dx.$ We have denoted by $|G|$ the volume of the group $G$ without assuming that the Haar measure is normalised. 
 
 We want to emphasize that when proving the weak (1,1) boundedness of $T,$ one has to estimate the quantity $|\{x\in G:|Tf(x)|>\gamma\alpha\}|$ only in the case where the Calder\'on-Zymund decomposition in \cite{CoifmanWeiss} is available, that is when $\alpha\gamma>\frac{1}{|G|}\smallint|f(x)|dx.$ Indeed, when estimating  $|\{x\in G:|Tf(x)|>\alpha\gamma\}|$ at any altitude $\alpha\gamma$ with $0<\alpha\gamma\leq \frac{1}{|G|}\smallint|f(x)|dx,$ one trivially has that
\begin{equation}
    |\{x\in G:|Tf(x)|>\alpha\gamma\}|\leq |G|\leq \frac{1}{\alpha\gamma}\smallint\limits_G|f(x)|dx,
\end{equation}as desired.

So, by fixing $\alpha\gamma$ as in \eqref{alphagamma}, note that in terms of $g,$ $b$ and $f$ one has the trivial estimate 
\begin{align*}
    |\{x:|Tf(x)|>\alpha\}|\leq |\{x:|Tg(x)|>\alpha/2\}|+|\{x:|Tb(x)|>\alpha/2\}|. 
\end{align*}
The estimates $\Vert g \Vert_{L^\infty}\lesssim \gamma \alpha$ and $\Vert g\Vert_{L^1}\lesssim \Vert f\Vert_{L^1},$ imply that
\begin{align*}
     \Vert g\Vert_{L^2}^2\lesssim\Vert g\Vert_{L^\infty}\Vert g\Vert_{L^1}\leq (\gamma \alpha)\Vert f\Vert_{L^1}. 
\end{align*}
So, by applying the Chebishev inequality and the $L^2$-boundedness of $T,$ we have 
\begin{align*}
   &|\{x:|Tg(x)|>\alpha/2\}|\lesssim 2^2\alpha^{-2}\Vert Tg \Vert_{L^2}\leq (2\Vert T \Vert_{\mathscr{B}(L^2)})^2\alpha^{-2}\Vert g\Vert_{L^2}^2 \\
  &\leq (2\Vert T \Vert_{\mathscr{B}(L^2)})^2\alpha^{-2}(\gamma \alpha)\Vert f\Vert_{L^1}\lesssim\Vert T \Vert_{\mathscr{B}(L^2)}^2\gamma\alpha^{-1}\Vert f\Vert_{L^1}\\
   &\lesssim_{\gamma} \alpha^{-1}\Vert f\Vert_{L^1}.
\end{align*}
 In what follows,  let us denote   $I^*=\bigcup I_j^*,$ where $I_j^*=B(x_j,2r_j),$  and let us make use of the doubling condition on $G$ in order to have the estimate
$$  |I^*_j|\sim  c^n|I_j|, $$ from which follows that 
$$ |I^*|\lesssim \sum_j|I_j^{*}| \lesssim \sum_j|I_j|\lesssim  \gamma^{-1}\alpha^{-1}\Vert f\Vert_{L^1}.   $$
Consequently, we have the estimates 
$$ |\{x:|Tb(x)|>\alpha/2\}|\leq |I^*|+ |\{x\in G\setminus I^* :|Tb(x)|>\alpha/2\}|  $$
$$ \leq ( 2\sqrt{n})^n \gamma^{-1}\alpha^{-1}\Vert f\Vert_{L^1}+|\{x\in G\setminus I^* :|Tb(x)|>\alpha/2\}| .  $$
$$ \lesssim_{\gamma} \alpha^{-1}\Vert f\Vert_{L^1}+|\{x\in G\setminus I^* :|Tb(x)|>\alpha/2\}| .  $$
So, to conclude the inequality \eqref{To:proof} we have to prove that 
\begin{equation}\label{To:proof:2}
  \sup_{\alpha>0}  \alpha|\{x\in G\setminus I^*:|Tb(x)|>\alpha/2\}|\leq C\Vert f\Vert_{L^1},
\end{equation}with $C$ independent of $f.$ So,  the proof of  Theorem \ref{main:th} consists of estimating the term $$|\{x\in G\setminus I^*:|Tb(x)|>\alpha/2\}|.$$ 

\begin{proof}[Proof of Theorem \ref{main:th}] 

We start by considering only the case $0<\theta<1.$ Indeed, the statement for $\theta=0$ in Theorem \ref{main:th} follows from the fundamental theorem of singular integrals due to Coifman and Weiss, see \cite[Theorem 2.4, Page 74]{CoifmanWeiss}.

From now, let us  suppose that the diameter of the support of $K$ is small, for instance, that $$\textnormal{diam}(\textnormal{supp}(K))<c,$$ where $0<c\leq 1$ is small enough. 
In particular, we can take $c$ small enough in order that we can guarantee that the exponential mapping
\begin{equation}\label{W:EXP}
    \omega=\textnormal{exp}:\nu\rightarrow \nu',\,\,{B}(e,c)\subset\nu',
\end{equation} is a diffeomorphism between two small neighborhoods $\nu'$ and $\nu$ of the identity element $e$ and of the origin $0\in \mathfrak{g}\cong \mathbb{R}^n,$ respectively.

For $\varepsilon>0$ define
\begin{equation}
    \phi(y,\varepsilon):=\frac{1}{|B(e,\varepsilon)|}1_{B(e,\varepsilon)}.
\end{equation} 
Now, for any $j,$ define
\begin{equation}
    \phi_{j}(y):=\phi(y,2^{-\frac{1}{1-\theta}}\textnormal{diam}(I_j)^{\frac{1}{1-\theta}}),\,\,\,
\end{equation}
\begin{equation}
   \tilde{b}_{j}:=b_{j}(\cdot)*\phi_{j},\,\,
\end{equation}and 
\begin{equation}
    \tilde{b}:=\sum_{j}\tilde{b}_{j}.
\end{equation}Note that
\begin{equation}\label{bjk:re:def}
  Tb=\sum_{j} T{b}_{j}.
\end{equation}
It is important to mention that in \eqref{bjk:re:def} the sums on the right hand side only runs over $j$ whith $\textnormal{diam}(I_j)< c .$ Indeed,  for all $x\in G\setminus I^*,$  the property of the support $\textnormal{diam}(\textnormal{supp}(K))< c ,$ implies that for all $j$ with $\textnormal{diam}(I_j)\geq c ,$ we have that $$Tb_{j}= b_{j}\ast K=0.$$ So  we only require to analyse the case where $\textnormal{diam}(I_j)< c .$ Indeed, for $x\in G\setminus I^*,$ and $j$ such that  $\textnormal{diam}(I_j)\geq c ,$
$$  b_{j}\ast K(x)=\smallint\limits_{I_j}K(y^{-1}x)b_{j}(y)dy.  $$
Because, in the integral above $x\in G\setminus I^*$ and $y\in I_j,$ $$|y^{-1}x|=\textnormal{dist}(x,y)>\textnormal{diam}(I_j)> c ,$$ we have that the element $y^{-1}x$ is not in the support of $K$ and then the integral vanishes.

In Figure \ref{Fig2}  we compare the size of $I_j$  with the size of the support of $\phi_j.$
\begin{figure}[h]
\includegraphics[width=8cm]{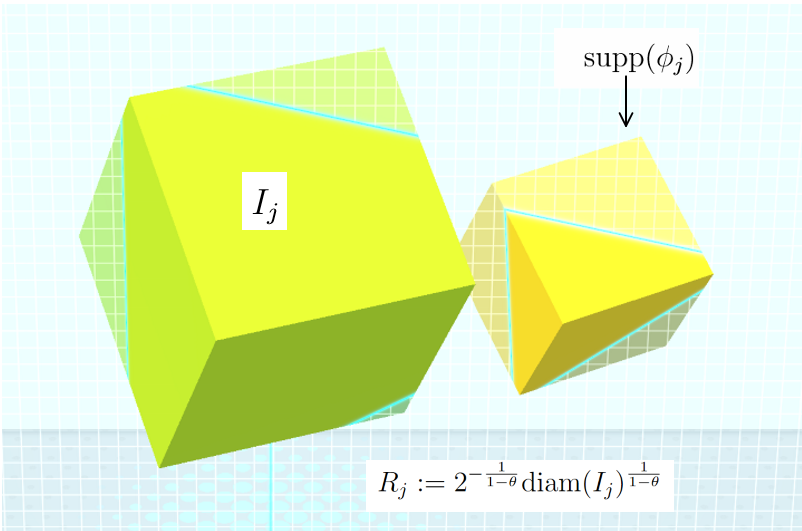}
\caption{ Let $j\in \mathbb{N}$. The replacement $\tilde b_{j}$ of $b_{j}$ is defined by  modifying its support.   To do so, we make the convolution between $b_j$ with the $L^1$-normalised function $\phi_j$ whose support is proportional to $R_j.$   }
 \label{Fig2}
\centering
\end{figure}
Now, going back to the analysis of \eqref{To:proof:2},
note that
\begin{align*}
  &|\{x\in G\setminus I^*:|Tb(x)|>\frac{\alpha}{2}\}|\\
  &\leq |\{x\in G\setminus I^*:|Tb(x)-T\tilde{b}(x)|>\frac{\alpha}{4}\}|+|\{x\in G\setminus I^*:|T\tilde b(x)|>\frac{\alpha}{4}\}|\\
  &\leq \frac{4}{\alpha}\Vert T(b-\tilde{b})\Vert_{L^1(G\setminus I^*)}+|\{x\in G\setminus I^*:|T\tilde b(x)|>\frac{\alpha}{4}\}|,
\end{align*}
and let us take into account the estimate:
\begin{align*}
    \Vert T(b-\tilde{b})\Vert_{L^1(G\setminus I^*)} &=\smallint\limits_{ G\setminus I^*}|Tb(x)-T\tilde{b}(x)|dx\\
    &\leq \sum_{j}\smallint\limits_{ G\setminus I^*}|Tb_{j}(x)-T\tilde{b}_{j}(x)|dx.
\end{align*} We are going to prove  that $T\tilde{b}$ and $T\tilde{b}_{j}$ are  good replacements for  $T{b}$ and $T{b}_{j},$ respectively, on the set $G\setminus I^*.$ 
Observe that 
\begin{align*}
  & \smallint\limits_{ G\setminus I^*}|Tb_{j}(x)-T\tilde{b}_{j}(x)|dx=\smallint\limits_{ G\setminus I^*}| b_{j}\ast K(x)-\tilde{b}_{j}\ast K(x)|dx \\
   &=\smallint\limits_{ G\setminus I^*}| b_{j}\ast K(x)-[{b}_{j}\ast \phi_j\ast K](x)|dx \\
   &=\smallint\limits_{ G\setminus I^*}\left|\smallint\limits_{I_j} K(y^{-1}x)b_{j}(y)dy-\smallint\limits_{I_j}(\phi_j\ast K)(y^{-1}x){b}_{j}(y)dy\right|dx \\
   &\leq\smallint\limits_{I_j} \smallint\limits_{ G\setminus I^*}| K(y^{-1}x)-\phi_j\ast K(y^{-1}x)|dx |b_j(y)|dy \\
    &\leq\smallint\limits_{I_j} \smallint\limits_{ |z|>\textnormal{diam}(I_j)}| K(z)-\phi_j\ast K
    (z)|dz |b_{j}(y)|dy\\
    &\leq\smallint\limits_{I_j} \smallint\limits_{ |z|>\textnormal{diam}(I_j)}| K(z)-\phi_j\ast K(z)|dz |b_{j}(y)|dy,
\end{align*}where, in the last line we have used the changes of variables $x\mapsto z=y^{-1}x,$ and then we observe that $|z|>\textnormal{diam}(I_j)$ when $x\in G\setminus I^*$ and $y\in I_j.$ Using that $\phi_j$ is   supported in a ball of radius  $$R_j:=2^{-\frac{1}{1-\theta}}\textnormal{diam}(I_j)^{\frac{1}{1-\theta}}$$ and that $\|\phi_j\|_{L^{\infty}}\leq 1/|B(e,R_j)|,$ we have that
\begin{align*}
   & \smallint\limits_{ |z|>\textnormal{diam}(I_j)  }| K(z)-\phi_j\ast K(z)|dz\\
   &=\smallint\limits_{ |z|>\textnormal{diam}(I_j)  }\left| K(z)\smallint\limits_{|y|<2^{-\frac{1}{1-\theta}}\textnormal{diam}(I_j)^{\frac{1}{1-\theta}}}\phi_{j}(y)dy-\smallint\limits_{|y|<2^{-\frac{1}{1-\theta}}\textnormal{diam}(I_j)^{\frac{1}{1-\theta}}}K(y^{-1}z)\phi_j(y)dy\right|dz\\
    &=\smallint\limits_{ |z|>\textnormal{diam}(I_j)  }\left|\, \smallint\limits_{|y|<2^{-\frac{1}{1-\theta}}\textnormal{diam}(I_j)^{\frac{1}{1-\theta}}}(K(z)-K(y^{-1}z))\phi_j(y)dy\right|dz\\
     &\leq   \smallint\limits_{|y|<2^{-\frac{1}{1-\theta}}\textnormal{diam}(I_j)^{\frac{1}{1-\theta}}} \smallint\limits_{ |z|>\textnormal{diam}(I_j)  }|K(y^{-1}z)-K(z)|dz|\phi_j(y)|dy\\
     &\leq \frac{1}{|B(e,R_j)|}  \smallint\limits_{|y|<R_j} \smallint\limits_{ |z|> 2R_j^{(1-\theta)}}|K(y^{-1}z)-K(z)|dz\,dy\\
     &\lesssim \frac{1}{R_j^n}  \smallint\limits_{|y|<R_j} \smallint\limits_{ |z|> 2R_j^{(1-\theta)}}|K(y^{-1}z)-K(z)|dz\,dy\\
     &\leq [K]_{H_{\infty,\theta}(G)},
\end{align*}where we have used that $2R_j^{1-\theta}=\textnormal{diam}(I_j).$ 
Now, the inequalities above allow us to finish the estimate of $\Vert T(b-\tilde{b})\Vert_{L^1(G\setminus I^*)}.$ Indeed,
\begin{align*}
    \Vert T(b-\tilde{b})\Vert_{L^1(G\setminus I^*)} &\leq \sum_{j} \smallint\limits_{I_j} \smallint\limits_{ G\setminus I^*}| K(y^{-1}x)-\phi_j\ast K(y^{-1}x)|dx |b_j(y)|dy \\
    &\leq\sum_{j}\smallint\limits_{I_j} \smallint\limits_{ |z|>\textnormal{diam}(I_j)} |K(z)-\phi_j\ast K(z)|dz |b_j(y)|dy\\
    &\lesssim [K]_{H_{\infty,\theta}(G)} \sum_{j}\smallint\limits_{I_j}|b_j(y)|dy\leq [K]_{H_{\infty,\theta}(G)}\Vert b \Vert_{L^1}\\
    &\lesssim [K]_{H_{\infty,\theta}(G)}\Vert f\Vert_{L^1}.
\end{align*}Putting together the estimates above we deduce that
\begin{align*}
    & |\{x\in G\setminus I^*:|Tb(x)|>\frac{\alpha}{2}\}|\leq   \frac{4}{\alpha}\Vert Tb-T\tilde{b}\Vert_{L^1(G\setminus I^*)}+|\{x\in G\setminus I^*:|T\tilde b(x)|>\frac{\alpha}{4}\}|\\
     &\lesssim \frac{4}{\alpha}[K]_{H_{\infty,\theta}(G)}\Vert f\Vert_{L^1}+|\{x\in G\setminus I^*:|T\tilde b(x)|>\frac{\alpha}{4}\}|.
\end{align*}Now, we will estimate the term in the right hand side of the previous inequality. Indeed, note that
$$
    |\{x\in G\setminus I^*:|T\tilde b(x)|>\frac{\alpha}{4}\}|\leq   |\{x\in G\setminus I^*:|T\tilde b(x)|^2>\frac{\alpha^2}{16}\}|
$$
$$\leq \frac{16}{\alpha^2}\Vert T\tilde{b}\Vert_{L^2}^2. 
$$ 
From now, let us consider the positive operator

\begin{equation}
    -\tilde{\Delta}=1+\mathcal{L}_G.
\end{equation}
Now, using  \eqref{decay} we deduce that $T(-\tilde{\Delta})^{\frac{n\theta}{4}}$ is bounded on $L^2,$ (see Remark \ref{Rem:Four:Conf}) and then
$$\Vert T\tilde{b}\Vert^2_{L^2}\leq \Vert T(-\tilde{\Delta})^{\frac{n\theta}{4}}\Vert_{\mathscr{B}(L^2)}^2\Vert(-\tilde{\Delta})^{-\frac{n\theta}{4}}\tilde{b}\Vert_{L^2}^2.$$

In order to estimate the $L^2$-norm, let us use the following lemma whose proof we postpone for a moment.
\begin{lemma}\label{lemma:Fefferman} The function $F:=(-\tilde{\Delta})^{-\frac{n\theta}{4}}\tilde{b}$ can be decomposed as the sums $F=F_1+F_2,$ where $  \Vert F_2\Vert_{L^2}^2\leq C\alpha\gamma\Vert f\Vert_{L^1},$ and $F_1$ is also a sum of  functions $F_{1}^{j}$ with the following property.
\begin{itemize}
    \item There exists $M_0\in \mathbb{N},$ and $A'>0,$ such that $F_{1}=\sum_{j:\textnormal{diam}(I_j)<1}F_{1}^{j}$,  $\Vert F_{1}^{j}\Vert_{L^2}^2\leq A'\alpha^2|I_j|,$ and for any $x\in \mathbb{R}^n,$ there at most $M_0$ values of $j$ such that  $F_1^j(x)\neq 0.$ 
\end{itemize}

\end{lemma}

Let us continue with the proof of Theorem \ref{main:th}. 
Using  Lemma \ref{lemma:Fefferman} and the inequalities in \eqref{I:j} we have that
\begin{align*}
\Vert T\tilde{b}\Vert_{L^2}^2&\lesssim \Vert F_1\Vert_{L^2}^2+\Vert F_2\Vert_{L^2}^2 \lesssim \alpha\gamma\Vert f\Vert_{L^1}+ \sum_{j}\Vert F_1^{j} \Vert_{L^2}^2\\
&\lesssim \alpha\gamma\Vert f\Vert_{L^1}+ \sum_j\alpha^2|I_j|\lesssim \alpha\gamma\Vert f\Vert_{L^1}+\alpha^2(\gamma\alpha)^{-1}\Vert f\Vert_{L^1}\\
    &= (\gamma^{-1}+\gamma)\alpha\Vert f\Vert_{L^1}.
\end{align*}
Consequently
$$
    |\{x\in G\setminus I^*:|T\tilde b(x)|>\frac{\alpha}{4}\}|\lesssim \frac{16}{\alpha^2}\Vert T\tilde{b}\Vert_{L^2}^2 
\lesssim (\gamma^{-1}+\gamma)\alpha^{-1}\Vert f\Vert_{L^1}.
$$ Thus, we have proved that
\begin{equation}
   |\{x\in \mathbb{R}^n:|Tf(x)|>\alpha\}|\leq C_{\gamma,[K]_{H_{\infty,\theta}(G)}}\alpha^{-1}\Vert f\Vert_{L^1},
\end{equation}with $C:=C_{\gamma,[K]_{H_{\infty,\theta}(G)}}$ independent of $f.$ Because $\gamma$ is fixed  we have proved the weak $(1,1)$ type of $T.$ Note that the $L^p$-boundedness of $T$ can be deduced from Marcinkiewicz interpolation theorem and the standard duality argument.
Thus, the proof is complete once we proved the statement in Lemma \ref{lemma:Fefferman}. To do so, let us consider the right-convolution kernel  $k_\theta:=(-\tilde{\Delta})^{-\frac{n\theta}{4}}\delta$ of the operator $(-\tilde{\Delta})^{-\frac{n\theta}{4}}$ and let us split the function $(-\tilde{\Delta})^{-\frac{n\theta}{4}}\tilde{b}(x)$ as follows:
\begin{align*}
   (-\tilde{\Delta})^{-\frac{n\theta}{4}}\tilde{b}(x)&=\sum_{j} (-\tilde{\Delta})^{-\frac{n\theta}{4}}\tilde{b}_{j}(x)=\sum_{j} \tilde{b}_{j}\ast k_\theta(x)\\
   &=\sum_{j:x\in I_j } \tilde{b}_{j}\ast k_\theta(x)+\sum_{j:x\nsim I_j} \tilde{b}_{j}\ast k_\theta(x)=:G_{1}(x)+G_{2}(x),
\end{align*}where $G_{1}(x):=\sum_{j:x\sim I_j} \tilde{b}_{j}\ast k_\theta(x)$ and $G_{2}(x):=\sum_{j:x\nsim I_j} \tilde{b}_{j}\ast k_\theta(x).$  We have denoted by $x\sim I_j,$ if $x$ belongs to $I_j$ or to some $I_{j'}$ with non-empty intersection with $I_j.$ By the properties of these sets there are at most $M_0$ sets $I_{j'}$ such that $I_j\cap I_{j'}\neq \emptyset.$ Also, the notation $x\nsim I_j$ will be employed to define the opposite of the previous property.

Let us prove the estimate
\begin{equation}\label{esti:G2}
    \Vert G_2\Vert_{L^2}^2\leq C\alpha\gamma\Vert f\Vert_{L^1}.
\end{equation} Observe that
\begin{align*}
   \Vert G_2\Vert_{L^1} &=\smallint\limits_G| \sum_{j:x\nsim I_j} {b}_{j}\ast \phi_j\ast k_\theta(x)|dx \leq \sum_{j:x\nsim I_j}\smallint\limits_G| {b}_{j}\ast \phi_j\ast k_\theta(x)|dx \\
   &\leq  \sum_{j}\| {b}_{j}\ast \phi_j\ast k_\theta\|_{L^1}\leq  \sum_{j}\| {b}_{j}\Vert_{L^1}\Vert \phi_j\ast k_\theta\|_{L^1}. 
\end{align*}Note that $(-\tilde{\Delta})^{-\frac{n\theta}{4}}$ is a pseudo-differential operator of order $-n\theta/2,$ and consequently, its kernel satisfies the estimate
\begin{equation*}
    |k_{\theta}(x)|\leq C|x|^{{-(-\frac{n\theta}{2}+n)}}\lesssim |x|^{-{n(1-\frac{\theta}{2})}},\,x\in G\setminus\{e\}.
\end{equation*}The condition $0<\theta<1,$ implies that $k_\theta$ is an integrable distribution and then
\begin{equation*}
  \Vert \phi_j\ast k_\theta\|_{L^1} \lesssim \Vert \phi_j\|_{L^1} \| k_\theta\|_{L^1}=\| k_\theta\|_{L^1}<\infty.   
\end{equation*}
So, we have that
\begin{align*}
     \Vert G_2\Vert_{L^1}\lesssim \sum_{j}\|b_j\|_{L^1}\lesssim\Vert f \Vert_{L^1}.
\end{align*}So, for the proof of \eqref{esti:G2}, and in view of the inequality $ \Vert G_2\Vert^2_{L^2}\leq \Vert G_2\Vert_{L^1} \Vert G_2\Vert_{L^\infty}$ is enough to show that $ \Vert G_2\Vert_{L^\infty}\lesssim\alpha\gamma\Vert f \Vert_{L^1}.$
To do this, let us consider $j$ such that  $x\nsim I_j.$ Since $\tilde{b}_{j}\ast k_\theta={b}_{j}\ast \phi_j\ast k_\theta,$ one has that
\begin{align*}
 | {b}_{j}\ast \phi_j\ast k_\theta(x)|&\leq \smallint\limits_{I_j}|\phi_j\ast k_\theta(y^{-1}x)||b_j(y)|dy\leq \sup_{y\in I_j}|\phi_j\ast k_\theta(y^{-1}x)|  \smallint\limits_{I_j}|b_j(y)|dy\\
 &=\sup_{y\in I_j}|\phi_j\ast k_\theta(y^{-1}x)| |I_j|\times \frac{1}{|I_j|} \smallint\limits_{I_j}|b_j(y)|dy.
\end{align*}
To continue, we follow as in \cite[Page 26]{Fefferman1970} the observation of Fefferman, that in view of the property $x\nsim I_j,$ we have that $\phi_j\ast k_\theta(y^{-1}x)$ is essentially constant over the ball $I_j=B(x_j,r_j)$ and we can estimate
\begin{equation}
    \sup_{y\in I_j}|\phi_j\ast k_\theta(y^{-1}x)| |I_j|\lesssim \smallint\limits_{I_j}|\phi_j\ast k_\theta(y'^{-1}x)|dy'.
\end{equation}On the other hand, observe that the positivity of the kernel $k_\theta,$ and of $\phi_j$ leads to
\begin{align*}
    \smallint\limits_{I_j}|\phi_j\ast k_\theta(y'^{-1}x)|dy' \frac{1}{|I_j|} \smallint\limits_{I_j}|b_j(y)|dy
    &=\smallint\limits_{G}|\phi_j\ast k_\theta(y'^{-1}x)| \left( \frac{1}{|I_j|} \smallint\limits_{I_j}|b_j(y)|dy\right)1_{I_{j}}(y')dy'\\
    &= \left(\frac{1}{|I_j|} \smallint\limits_{I_j}|b_j(y)|dy\times1_{I_j}\right)\ast\phi_j\ast k_\theta(x).
\end{align*}Consequently, we have
\begin{align*}
    |G_2(x)|&\leq \sum_{j:x\nsim I_j} |\tilde{b}_{j}\ast k_\theta(x)|\lesssim  \sum_{j:x\nsim I_j}\left(\frac{1}{|I_j|} \smallint\limits_{I_j}|b_j(y)|dy\times1_{I_j}\right)\ast\phi_j\ast k_\theta(x)\\
    &\lesssim\sum_{j:x\nsim I_j}\gamma\alpha\times1_{I_j}\ast\phi_j\ast k_\theta(x)=\smallint\limits_{G}\sum_{j:x\nsim I_j}\gamma\alpha\times1_{I_j}\ast\phi_j(z) k_\theta(z^{-1}x)dz\\
    &\lesssim \gamma\alpha \Vert k_\theta\Vert_{L^1}\left\Vert \sum_{j:x\nsim I_j}1_{I_j}\ast\phi_j \right\Vert_{L^{\infty}}.
\end{align*} By observing that  the supports of the functions $1_{I_j}\ast\phi_j $'s have bounded overlaps we have that $\left\Vert \sum_{j:x\nsim I_j}1_{I_j}\ast\phi_j \right\Vert_{L^{\infty}}<\infty,$ and that $\Vert G_2\Vert_{L^\infty}\lesssim\gamma\alpha.$ 

It remains only to prove that $\Vert G_1\Vert_{L^2}^2\lesssim \alpha \Vert f\Vert_{L^1}.$ Let us define
\begin{equation}\label{F:j:iproof}G^{j}(x):= \begin{cases}{b}_{j}\ast \phi_j\ast k_\theta(x),& \text{ }x\in I_j,
\\0 ,& \,\text{otherwise. } \end{cases}
\end{equation}Then $G_1=\sum_{j} G^{j}$ and in view of the finite overlapping of the balls $I_{j}$'s, there is $M_0\in \mathbb{N},$ such that for any $x\in G,$ $G^{j}(x)\neq 0,$ for at most $M_0$ values of $j.$ Therefore, we have that
\begin{align*}
    \smallint_G|G_1(x)|^2dx &\leq M_0 \sum_{j}\smallint_G|G^{j}(x)|^2dx =\sum_{j}\smallint_{I_j}|{b}_{j}\ast \phi_j\ast k_\theta(x)|^2dx\\
    &\leq \sum_{j}\Vert {b}_{j}\Vert_{L^1}^2\Vert \phi_j\ast k_\theta\Vert_{L^2}^2\leq \sum_{j}\alpha^2\gamma^2|I_j|^2\Vert \phi_j\ast k_\theta\Vert_{L^2}^2.
\end{align*} 
Because our argument is purely local let us use the diffeomorphism $\omega=\textnormal{exp}:\nu\rightarrow \nu'$ in \eqref{W:EXP}. Because ${B}(e,c)\subset\nu',$   we can estimate the norm $\Vert \phi_j\ast k_\theta\Vert_{L^2}$ in a small neighbourhood $\nu'$ of the identity element $e\in G.$  Note that for any $j$ with $\textnormal{diam}(I_j)< c ,$ we have the inclusion $B(e,R_j)\subset \nu'.$ So, we have
$$ \Vert \phi_j\ast k_\theta\Vert_{L^2}^2=\Vert(1+\mathcal{L}_G)^{-\frac{n\theta}{4}}\phi_j\Vert^2_{L^2}=\Vert\phi_j\Vert_{H^{-n\theta/2}}\asymp \Vert(1+\mathcal{L}_{\mathbb{R}^n})^{-\frac{n\theta}{4}}{\phi}'_j\Vert_{L^2}^2, $$
where $$\phi_j'=\phi_j\circ \omega=\frac{1}{B(e,R_j)}1_{B(e,R_j)}\circ \omega,\,\,\textnormal{supp}(\phi_j')\subset \nu.$$

Note that $(1+\mathcal{L}_{\mathbb{R}^n})^{-\frac{n\theta}{4}}$ is bounded on $L^2,$ and clearly $\phi_j'\in L^2,$ so that $\phi_j'\in \textnormal{Dom}((1+\mathcal{L}_{\mathbb{R}^n})^{-\frac{n\theta}{4}}).$ Also, 

\begin{equation}
    \textnormal{supp}[\phi_j']\subset{B}(0,R_j'),\,R_{j}'\sim R_{j}.
\end{equation}
Define for any $j,$ $$\psi_{j}'(X):=|B(0,R_j')|\phi_{j}'(R_j'X),\,X\in \nu.$$
Then, one has the identity
\begin{equation*}
    \phi_j'(X)=\frac{1}{|B(0,R_j')|}\psi_j'\left(\frac{X}{R_j'}\right),\,X\in \nu.
\end{equation*} Observe that $|B(0,R_j')|=v_{n}{R_{j}'}^n,$ where $v_n$ is the volume of the unite ball $B(0,1).$ Also, note that $\Vert \psi_j'\Vert_{L^\infty}=1,$ and $\textnormal{supp}(\psi_j')\subset B(0,1).$  Using the Plancherel theorem, we have that
\begin{align*}
  \Vert(1+\mathcal{L}_{\mathbb{R}^n})^{-\frac{n\theta}{4}}{\phi}_j'\Vert_{L^2}&=\Vert (1+|\eta|^2)^{-\frac{n\theta}{4}}\mathscr{F}_{\mathbb{R}^n}[\phi_j'](\eta)\Vert_{L^2} \\
  &= \Vert (1+|\eta|^2)^{-\frac{n\theta}{4}}\mathscr{F}_{\mathbb{R}^n}\left[\frac{1}{|B(0,R_j')|}\psi_j'\left(\frac{\cdot}{R_j'}\right)\right]\left(\eta\right)\Vert_{L^2} \\
  &= (1/v_n)\Vert (1+|\eta|^2)^{-\frac{n\theta}{4}}\mathscr{F}_{\mathbb{R}^n}\left[\frac{1}{{R_j'}^n}\psi_j'\left(\frac{\cdot}{R_j'}\right)\right]\left(\eta\right)\Vert_{L^2} \\
  &= (1/v_n) \Vert (1+|\eta|^2)^{-\frac{n\theta}{4}}\mathscr{F}_{\mathbb{R}^n}[\psi_j']\left(R_j'\eta\right)\Vert_{L^2}.
\end{align*}
Observe that
\begin{align*}
    \Vert (1+|\eta|^2)^{-\frac{n\theta}{4}}\mathscr{F}_{\mathbb{R}^n}[\psi_j']\left(R_j\eta\right)\Vert_{L^2}^2&=\smallint\limits|(1+|\eta|^2)^{-\frac{n\theta}{4}}\mathscr{F}_{\mathbb{R}^n}[\psi_j']\left(R_j\eta\right)|^2d\eta\\
    &={R'_{j}}^{-n}\smallint\limits|(1+|{R'_j}^{-1}z|^2)^{-\frac{n\theta}{4}}\mathscr{F}_{\mathbb{R}^n}[\psi_j'](z)|^2dz\\
    &= {R'_{j}}^{-n}{R'_{j}}^{n\theta}\smallint\limits|(R_{j}'+|z|^2)^{-\frac{n\theta}{4}}\mathscr{F}_{\mathbb{R}^n}[\psi_j'](z)|^2dz\\
    &\leq {R'_{j}}^{-n}{R'_{j}}^{n\theta}\smallint\limits|(|z|^2)^{-\frac{n\theta}{4}}\mathscr{F}_{\mathbb{R}^n}[\psi_j'](z)|^2dz\\
    &\asymp {R'_{j}}^{-n}{R'_{j}}^{n\theta}\smallint\limits|(-\Delta)^{-\frac{n\theta}{4}}\psi_j'(x)|^2dx\\
    &= R_{j}^{-n(1-\theta)}\Vert (-\Delta)^{-\frac{n\theta}{4}} \psi_j'\Vert_{L^2}^2.
\end{align*}
By the compactness of $G,$ for all $\varepsilon>0,$ there exists a finite number of elements $x_{0}^\varepsilon=e_{G},x_{k}^\varepsilon,$ $1\leqslant k\leqslant N_0(\varepsilon),$ in $G,$ and some smooth functions $\chi_{k}^\varepsilon\in C^\infty(G,[0,1]),$ supported in $B(e_G,\frac{\varepsilon}{2}),$ such that
\begin{equation}\label{chij}
    G=\bigcup_{k=1}^{N_0(\varepsilon)}B(x_{k}^\varepsilon,\varepsilon/4),\,\,\,\textnormal{   and  }\sum_{j=0}^{N_0(\varepsilon)}\chi_{k}^\varepsilon(x_k^{-1} x)=1,\,\,\, x\in G.
\end{equation}
Note that  for $\varepsilon>0$ small enough, the support of any $x\mapsto\chi_{k}^\varepsilon(x_k^{-1} x)$ is inside of $\nu'.$ In consequence, we can estimate
\begin{align*}
    \|(-\Delta)^{-\frac{n\theta}{4}}\psi_j'\|_{L^2}& \asymp \|(-\mathcal{L}_G)^{-\frac{n\theta}{4}}\psi_j'\circ \textnormal{exp}\|_{L^2}\leq \sum_{j=0}^{N_0(\varepsilon)}\|(-\mathcal{L}_G)^{-\frac{n\theta}{4}}[\psi_j'\circ \textnormal{exp}\cdot \chi_{k}^\varepsilon(x_k^{-1} \cdot)]\|_{L^2}\\
    &\lesssim  \sum_{j=0}^{N_0(\varepsilon)}\|\psi_j'\circ \textnormal{exp}\cdot \chi_{k}^\varepsilon(x_k^{-1} \cdot)\|_{L^2} \asymp  \|\psi_j'\circ \textnormal{exp}\|_{L^2} \asymp  \|\psi_j'\|_{L^2},
\end{align*}
where we have used the $L^2$-boundedness of the operator $(-\mathcal{L}_G)^{-\frac{n\theta}{4}}$ in view of the Stein identity (see \cite[Page 58]{SteinIdentity})    $$(-\mathcal{L}_G)^{-\frac{n\theta}{4}}=\left(-\mathcal{L}_G+\textnormal{Proj}_{\textnormal{Ker}(-\mathcal{L}_G)}\right)^{-\frac{n\theta}{4}}-\textnormal{Proj}_{\textnormal{Ker}(-\mathcal{L}_G)},$$ and of Remark 2.4 of \cite{RuzhanskyWirth2015}. Indeed, we have that $(-\mathcal{L}_G)^{-\frac{n\theta}{4}}$ is a pseudo-differential operator on $G$ of order $-n\theta/2$ and hence bounded on $L^2(G).$ Note that, this argument does not apply in the case of $\mathbb{R}^n$ because the spectrum of the Euclidean Laplacian is continuous and the negative powers of $-\Delta_{\mathbb{R}^n}$ are not pseudo-differential operators while the powers of  $1-\Delta_{\mathbb{R}^n}$ are pseudo-differential operators with symbols in the Kohn-Nirenberg classes, see e.g. Taylor \cite{Taylorbook1981}. 

Continuing with the proof, since $\Vert\psi_j\Vert_{L^\infty}=1$ and $\psi_j$ is supported in the unit ball, we have that $\Vert \psi_j\Vert_{L^2}\leq 1.$ Consequently,
$$ \Vert \phi_j\ast k_\theta\Vert_{L^2}^2\lesssim R_{j}^{-n(1-\theta)}\sim |I_j|^{-1}. $$
Finally, we deduce that
\begin{align*}
    \smallint_G|G_1(x)|^2dx &\leq \sum_{j}\alpha^2\gamma^2|I_j|^2\Vert \phi_j\ast k_\theta\Vert_{L^2}^2\lesssim \sum_{j} \alpha^2\gamma^2|I_j|^2|I_j|^{-1}=\alpha^2\gamma^2\sum_{j}|I_j|\\
    &\lesssim \alpha\gamma\Vert f\Vert_{L^1}\lesssim_\gamma \alpha\Vert f\Vert_{L^1},
\end{align*} in view of \eqref{I:j}. Consequently, we can take $F_2:=G_2,$ $F_1:=G_1$ and  $F_{1}^{j}:=G_{1}^{j}.$ Thus, the proof of  Lemma \ref{lemma:Fefferman} is complete as well as the proof of Theorem  \ref{main:th}.
\end{proof}

\section{Examples}\label{Examples:t:su2:OFM}
In this section we give some applications of Theorem \ref{main:th}  on the torus, on $\textnormal{SU}(2)\cong \mathbb{S}^3$ and for oscillating multipliers. We refer the reader to \cite[Section 4]{CR20221} for a variety of examples on the boundedness of oscillating singular integrals on Lie groups of polynomial growth.
\subsection{Oscillating integrals on the torus $\mathbb{T}^n$}
 Let  $G=\mathbb{T}^n\equiv \mathbb{R}^n/\mathbb{Z}^n,$ be the $n$-torus.  In this case we have the identification $\widehat{\mathbb{T}}^n=\{e_{\ell}\}_{\ell\in \mathbb{Z}^n}\sim \mathbb{Z}^{n}$ for its unitary dual. We have denoted $e_{\ell}$ to the exponential function $e_{\ell}(x)=e^{i2\pi \ell\cdot x},$ $x=(x_1,\cdots ,x_n)\in \mathbb{T}^n.$

In terms of the Fourier transform of a distribution $K$ on the torus
\begin{equation}
    \widehat{K}(\ell):=\smallint_{\mathbb{T}^n}e_{-\ell}(x)K(x)dx,\,\,\ell \in \mathbb{Z}^n,
\end{equation}the Fourier transform condition \eqref{LxiG} becomes equivalent to the estimate
\begin{equation}\label{LxiT:n}
  \forall \ell\in \mathbb{Z}^n,\,\,  | \widehat{K}(\ell)|\leq C\langle \ell\rangle^{-\frac{n\theta}{2}},\,\ \langle \ell\rangle:=(1+4\pi^2|\ell|^2)^{\frac{1}{2}}\sim |\ell|:=\sqrt{\ell_1^2+\cdots+ \ell_n^2}.
\end{equation}Note that the kernel condition \eqref{GS:CZ:cond} takes the form
\begin{equation}\label{Con:kern:torus}
    [K]_{H_{\infty,\theta}}=\sup_{R>0}\sup_{|y|<R} \smallint\limits_{|x|\geq 2R^{1-\theta}}|K(x-y)-K(x)|dx dy  <\infty.
\end{equation}In view of Theorem \ref{main:th}, a convolution operator $T$ associated to a convolution kernel $K,$ satisfying the Fourier transform condition \eqref{LxiT:n} and the smoothness condition  \eqref{Con:kern:torus}   is of weak (1,1) type, that is $T:L^{1}(\mathbb{T}^n)\rightarrow L^{1,\infty}(\mathbb{T}^n)$ is bounded. As the referee of this manuscript pointed out, in the  case of the torus $\mathbb{T}^n$ this continity result should also be just a special case of the known result due to Fefferman \cite{Fefferman1970} on $\mathbb{R}^n$ for a  distributions  $K$ with its support small enough as a consequence of a periodisation argument (see e.g. \cite[Chapters III and IV]{Ruz}). 

\subsection{Oscillating integrals on $\textnormal{SU}(2)\cong \mathbb{S}^3$}\label{SU2}Let us consider the compact Lie group of complex unitary $2\times 2$-matrices  
$$ \textnormal{SU}(2)=\{X=[X_{ij}]_{i,j=1}^{2}\in \mathbb{C}^{2\times 2}:X^{*}=X^{-1}\},\,X^*:=\overline{X}^{t}=[\overline{X_{ji}}]_{i,j=1}^{2}. $$
Let us consider the left-invariant first-order  differential operators $$\partial_{+},\partial_{-},\partial_{0}: C^{\infty}(\textnormal{SU}(2))\rightarrow C^{\infty}(\textnormal{SU}(2)),$$ called creation, annihilation, and neutral operators respectively, (see Definition 11.5.10 of \cite{Ruz}) and let us define 
\begin{equation*}
    X_{1}=-\frac{i}{2}(\partial_{-}+\partial_{+}),\, X_{2}=\frac{1}{2}(\partial_{-}-\partial_{+}),\, X_{3}=-i\partial_0,
\end{equation*}where $X_{3}=[X_1,X_2].$ The system $X=\{X_1,X_2,X_3\}$ is an orthonormal basis of the Lie algebra $\mathfrak{su}(2)$ of $\textnormal{SU}(2),$ and its positivie Laplacian is given by
\begin{equation}
    \mathcal{L}_{\textnormal{SU}(2)}=-X_1^2-X_2^2-X_3^2=-\partial_0^2-\frac{1}{2}[\partial_+\partial_{-}+\partial_{-}\partial_{+}].
\end{equation}
We record that the unitary dual of $\textnormal{SU}(2)$ (see \cite{Ruz}) can be identified as
\begin{equation}
\widehat{\textnormal{SU}}(2)\equiv \{ [t_{l}]:2l\in \mathbb{N}, d_{l}:=\dim t_{l}=(2l+1)\}\sim \frac{1}{2}\mathbb{N}.
\end{equation}
There are explicit formulae for $t_{l}$ as
functions of Euler angles in terms of the so-called Legendre-Jacobi polynomials, see \cite{Ruz}. Again, by following e.g.  \cite{Ruz},  the spectrum of the positive Laplacian $\mathcal{L}_{\textnormal{SU}(2)}$ can be indexed by the sequence
$$\lambda_\ell:= \ell(\ell+1),\quad \ell\in \frac{1}{2}\mathbb{N}.$$
Because $n=\dim(\textnormal{SU}(2))=3,$ the Fourier transform condition \eqref{LxiG}  takes the form
 \begin{equation}\label{Fouriergrowth:2SU2}
      \exists C>0,\,\forall  \ell\in\frac{1}{2} \mathbb{N},\quad \Vert \widehat{K}(\ell) \Vert_{\textnormal{op}}\leq C(1+\ell(\ell+1))^{-\frac{3\theta}{4}}\sim (1+\ell)^{-\frac{3\theta}{2}},
    \end{equation}
where $\widehat{K}(\ell)=\smallint_{\textnormal{SU}(2)}K(Z)t_{\ell}(Z)^*dZ,$ is the Fourier transform of the group of $K.$
In view of Theorem \ref{main:th}, if  
    $K$ satisfies \eqref{Fouriergrowth:2SU2} and the kernel condition 
    \begin{equation}\label{GS:CZ:cond:22'SU2}
        [K]_{H_{\infty,\theta}}:=\sup_{R>0}\sup_{|Y|<R} \smallint\limits_{|X|\geq 2R^{1-\theta}}|K(Y^{-1}X)-K(X)|dX dY  <\infty,
    \end{equation} then $T$ is of weak (1,1) type, that is  $T:L^1(\textnormal{SU}(2))\rightarrow L^{1,\infty}(\textnormal{SU}(2))$ extends to a bounded operator. In \eqref{GS:CZ:cond:22'SU2},  $|X|$ denotes the  norm of  $X\in \textnormal{SU}(2)$ with respect to the geodesic distance on $\textnormal{SU}(2)$.
    
\subsection{Application to oscillating multipliers}\label{Sub:sec:L} Let us illustrate our main Theorem \ref{main:th} in the context of oscillating Fourier multipliers. On $G,$ the prototype of a oscillating singular integral is the spectral multiplier of  $\mathcal{L}_G,$ given by
\begin{equation}\label{t:theta:clg}
    T_{\theta}(\mathcal{L}_G):=(1+\mathcal{L}_G)^{-\frac{n\theta}{4}}e^{i (1+\mathcal{L}_G)^{\frac{\theta}{2}}}, \, \,0\leq \theta<1.
\end{equation}Indeed, $T_{\theta}(\mathcal{L}_G)$ is a convolution operator with right-convolution kernel $K=K_\theta$ whose Fourier transform is given by
$$
    \widehat{K}(\xi)=\langle\xi\rangle^{-\frac{n\theta}{2}}{e^{i\langle\xi\rangle^\theta}},\,\langle\xi\rangle:=(1+\lambda_{[\xi]})^{\frac{1}{2}},\, [\xi]\in \widehat{G}.$$ Let $d(x,y)$ be the geodesic distance on $G$ induced by the Riemannian metric $g$. It was proved by Chen and Fan in \cite{Chen}  that $K(x)$ behaves essentially as $ c_nd(x,e)^{-n}e^{ic_{n}'d(x,e)^{\theta'}},$ where
 $ \theta'=\frac{\theta}{\theta-1},$ and consequently that $$|\nabla K(x)|=|(X_1 K(x),\cdots, X_nK(x))|\lesssim d(x,e)^{-n-1+\theta'},$$ from which it is well known that $K$ satisfies \eqref{GS:CZ:cond}, that is,
 $$   [K]_{H_{\infty,\theta}}:=\sup_{0<R\leq 1}\sup_{|y|\leq R}  \smallint\limits_{|x|\geq 2R^{1-\theta}}|K(y^{-1}x)-K(x)|dx  <\infty, $$
justifying the term {\it `oscillating'} for this family of multipliers. Using a smooth cut-off function $\psi$ with small compact support (for instance assume that $\textnormal{diam}(\textnormal{supp}(\psi))<1$), one can write the operator $T_\theta(\mathcal{L}_G)=T_1+T_2,$ with $T_1$ associated to the right-convolution kernel $K_1=\psi K,$ and $K_2=K-K_1,$ and then $T_2$ is bounded on $L^1(G).$ In view of Theorem \ref{main:th}, one can use that $T_1$ is of weak (1,1) type, and then one can deduce that $T_\theta(\mathcal{L}_G)$ extend to an operator of weak (1,1) type. 
\begin{remark}\label{Coordinates:approach}
By following the analysis of pseudo-differential operators under local coordinates systems, one can prove that for $0\leq \theta<\frac{1}{2},$ the operator $T_{\theta}(\mathcal{L}_G)$ belongs to the H\"ormander class $\textnormal{Op}(S^{m}_{1-\theta,\theta})$ on $G,$ with the order $m=-\frac{n\theta}{2},$ see e.g the book of M. Taylor \cite{Taylorbook1981} for details.

By microlocalising the Fefferman weak (1,1) estimate in \cite{Fefferman1970}, we deduce that for any $0\leq \theta<\frac{1}{2},$ $T_{\theta}(\mathcal{L}_G)$ extends to an operator of weak $(1,1)$ type. This is consequence of the fact that for all  $0\leq \theta<\frac{1}{2},$ the class $\textnormal{Op}(S^{m}_{1-\theta,\theta})$  is invariant under changes of coordinates. However, the range $\frac{1}{2}\leq \theta<1$ can be covered by our Theorem \ref{main:th}. 

We end this remark by observing that the boundedness of \eqref{t:theta:clg} from the Hardy space $H^1(G)$ into $L^1(G)$ has been proved by Chen and Fan in \cite{Chen} in the complete range $0\leq \theta<1$.
\end{remark}
\begin{remark}
Other generalisations on compact Lie groups of estimates for oscillating integrals to pseudo-differential operators have been considered in the papers \cite{Cardona2,RuzhanskyDelgado2017} and \cite{RuzhanskyWirth2015}. We refer the reader to \cite[Chapter V]{CR20} for Fourier transform conditions of Fourier multipliers of weak $(1,1)$ type on compact Lie groups in subelliptic settings.
\end{remark}

\bibliographystyle{amsplain}

\end{document}